\documentclass[review]{elsarticle}
\usepackage[usenames]{color}
\usepackage{extarrows}
\usepackage{lineno,hyperref}
\modulolinenumbers[5]

\journal{Journal of \LaTeX\ Templates}

\usepackage{amsfonts}
\usepackage{amsmath, cases}
\usepackage{indentfirst}
\usepackage{graphicx}
\usepackage{latexsym,bm, amsthm}

\usepackage{epsfig}
\usepackage{latexsym}
\usepackage{amssymb}

\begin{document}

\begin{frontmatter}

\title{Weak group inverse }

\author[1-address,2-address]{Hongxing Wang\corref{mycorrespondingauthor}}
\cortext[mycorrespondingauthor]{Corresponding author}
\ead{winghongxing0902@163.com}
\address[1-address]{
College of Science,
Guangxi University for Nationalities,
Nanning, 530006, P.R. China }

\author[2-address]{Jianlong Chen}
 \ead{jlchen@seu.edu.cn)}
\address[2-address]{
School of Mathematics,
Southeast University, Nanjing, 210096, P.R. China}

\begin{abstract}
In this paper,
we introduce a weak group inverse
(called the WG inverse in the present paper)
for square matrices of an arbitrary index,
and
give some of its characterizations and properties.
Furthermore,
we introduce two orders:
one is a pre-order and the other is a partial order,
and derive several characterizations of the two orders.
At last,
one characterization  of the core-EP order
is derived by using   the WG inverses.
\end{abstract}


\end{frontmatter}

\linenumbers

\section{Introduction}
\numberwithin{equation}{section}  
\newtheorem{theorem}{T{\scriptsize HEOREM}}[section]
\newtheorem{lemma}[theorem]{L{\scriptsize  EMMA}}
\newtheorem{corollary}[theorem]{C{\scriptsize OROLLARY}}
\newtheorem{proposition}[theorem]{P{\scriptsize ROPOSITION}}
\newtheorem{remark}{R{\scriptsize  EMARK}}[section]
\newtheorem{definition}{D{\scriptsize  EFINITION}}[section]
\newtheorem{algorithm}{A{\scriptsize  LGORITHM}}[section]
\newtheorem{example}{E{\scriptsize  XAMPLE}}[section]
\newtheorem{problem}{P{\scriptsize  ROBLEM}}[section]
\newtheorem{assumption}{A{\scriptsize  SSUMPTION}}[section]

In this paper, we use the following notations.
The symbol $ {\mathbb{C}}_{m, n}$ is
the set of $m\times n$ matrices with complex entries;
 $A^\ast  $, ${\mathcal{R}}(A)$ and ${\rm rk}\left( A \right)$
 represent
the \emph{conjugate transpose},
\emph{range space}  (or \emph{column space}) and \emph{rank}
of $A \in {\mathbb{C} }_{m, n} $.
Let $A\in\mathbb{C}_{n, n}$,
  the smallest positive integer $k$,
   which satisfies ${\rm rk}\left( A^{k+1} \right)={\rm rk}\left( A^k \right)$,
 is called the \emph{index} of $A$ and is  denoted as  ${\rm Ind}(A)$.
%The index of a non-singular matrix  is $0$ and the index
%of a null matrix is $1$.
%
%
The symbol  $\mathbb{C}^{\mbox{\rm\footnotesize\texttt{CM}}}_{n }$
stands for the set of $n\times n$  matrices
of index  equal to one.
The {\emph{Moore-Penrose inverse}} of $A\in\mathbb{C}_{m, n}$ is defined as
    the unique matrix  $X\in\mathbb{C}_{n, m}$
    satisfying  the equations:
    \begin{align*}
 (1)~AXA = A, \ \  (2)~XAX = X, \ \
 (3)~\left( {AX} \right)^\ast = AX, \ \
  (4)~\left( {XA} \right)^\ast = XA,
 \end{align*}
and is denoted as $X = A^{\dag}$;
if $X$ satisfies the equation $AXA=A$,
then $X$ is called a \emph{g-inverse} of $A$,
and is denoted as $A^-$;
$E_A$  stands for the one
orthogonal projection $E_A=I-AA^\dag$.
The \emph{Drazin inverse}  of $A \in {\mathbb{C} }_{n,n}$ is defined as
    the unique matrix $X\in\mathbb{C}_{n, n}$ satisfying the equations
\begin{align*}
(1^k)~XA^{k+1} = A^k, \ \ (2)~ XAX = X, \ \ (5)~ AX = XA,
\end{align*}
and is usually denoted as $X = A^D$.
In  particular,
 when  $A\in \mathbb{C}^{\mbox{\rm\footnotesize\texttt{CM}}}_{n }$,
 the matrix  $X$   is called the \emph{group inverse} of $A$,
 and is denoted as $X = A^\#$ ({see \rm\cite{Ben2003book}}).
The \emph{core inverse} of $A \in {\mathbb{C} }_{n}^{\mbox{\rm\footnotesize\texttt{CM}}}$
is defined as the unique matrix $X\in\mathbb{C}_{n,n}$ satisfying
\begin{align}
\nonumber
%\label{2-3}
AX = AA^\dag,~~  {\mathcal{R}}\left( {X} \right) \subseteq  {\mathcal{R}}\left( A\right)
\end{align}
and is denoted as $X = A^{\tiny\textcircled{\#}}$ {\rm\cite{Baksalary2010lma}}.
When $A\in\mathbb{C}^{\mbox{\rm\footnotesize\texttt{CM}}}_{n }$,
we call it a  core invertible (or group invertible) matrix.

Recently,
the research of the core inverse and related problems is drawing ever-growing attention.
Several generalized core inverses are introduced,
which are the {DMP inverse},
the {B-T inverse}
and
the {core-EP inverse} \cite{Baksalary2014amc,Malik2014amc,Manjunatha2014lma}.
Let $A \in {\mathbb{C} _{n,n} }$ with ${\rm Ind}\left( A \right) = k$.
The \emph{DMP inverse} of $A  $ is
$A^{d,\dag}=A^DAA^\dag$
{\cite{Malik2014amc}}.
The \emph{B-T inverse} of $A$ is %defined as
$A^\diamondsuit=\left(A^2A^\dag\right)^\dag$
\cite[Definition 1]{Baksalary2014amc}.
The \emph{core-EP inverse} of $A $ is
$A^{\tiny\textcircled{\dag}}
= A^k\left( {\left( {A^\ast  } \right)^kA^{k + 1}} \right)^ -A^k$
 \cite[Theorem 3.5 and Remark 2]{Manjunatha2014lma}.
Especially,
when $A\in { \mathbb{C}}_n^{\mbox{\rm\footnotesize\texttt{CM}}} $,
$A^{\tiny\textcircled{\#}}
=A^\diamondsuit
=A^{d,\dag}
=A^{\tiny\textcircled{\dag}}$
\cite{Baksalary2014amc,Malik2014amc,Manjunatha2014lma}.
The relevant orders are presented,
for example,
 the core-EP order, the DMP order and the B-T order \cite{Baksalary2014amc,Deng2015amc,Wang2016laa-A}.
 The three orders   are all pre-orders,  although the core order is a partial order.

 In \cite{Wang2016laa-A},
 Wang introduced the core-EP decomposition.
Applying the decomposition,
 Wang introduced the core-minus partial order,
 in a way similar to
 applying the core-nilpotent  decomposition to define  the C-N partial order.

Furthermore,
it is   known that
the index of group invertible matrix
is also  equal to one,
that is,
one matrix is core invertible if and only if it is group invertible.
Although  the generalizations of the core inverse attract much attention,
the generalizations of the group inverse get little.
Therefore,
it is of interest to inquire
whether
 the group inverse can be generalized
by some decompositions.

In this paper,
our main tools are two decompositions:
one   is the core decomposition,
the other is the core-EP decomposition.
The aim of the paper is to
introduce a generalized group inverse,
consider its applications
and
derive some of its characterizations and properties.

\section{Preliminaries}
\label{Preliminaries}

In this section,
we present some preliminary results.

 \begin{lemma}
{\rm\cite{Ben2003book}} %[Lemma 2]
 Let
$A\in \mathbb{C}_{n,n}$
be with
${\rm Ind}(A)=k$.
Then
\begin{align}
\label{2-7}
A^D
=
A^{k }\left(A^{ k+1 }\right)^{\#}.
\end{align}
\end{lemma}

 \begin{lemma}
{\rm\cite{Baksalary2010lma,Hartwig1984lma,Wang2016laa-A}} %[Lemma 2]
 Let
$A\in \mathbb{C}_{n,n}$
be  with
${\rm Ind}(A)=k$.
Then
there exists  a   unitary matrix $U$
 such that
\begin{align}
\label{2-1}
A = U\left[ {{\begin{matrix}
 \Sigma K   & \Sigma L   \\
 0   & 0   \\
\end{matrix} }} \right]U^\ast,
\end{align}
where
$\Sigma=\mbox{diag}\left( {\sigma_1 I_{r_1} ,\sigma _2 I_{r_2} ,\ldots , \sigma _t I_{r_t} } \right)$
is the diagonal matrix of singular values of $A$,
$\sigma_1 > \sigma _2 > \cdots > \sigma _t >0$,
$r_1 +r _2 + \cdots + r _t =r$,
and $K \in {\mathbb{C} }_{r,r} $,
$L \in {\mathbb{C} }_{r, n - r} $
satisfy
$KK^\ast  + LL^\ast  =I_r $.

Furthermore,
$A $ is core invertible
 if and only if
 ${\Sigma K}$  is   non-singular.
When  $A\in\mathbb{C}^{\mbox{\rm\footnotesize\texttt{CM}}}_{n }$,
 (\ref{2-1}) is called the  {core decomposition} of $A$
 and
\begin{align}
\label{2-2}
A^{\tiny\textcircled{\#}}
&
=
 U\left[ {{\begin{matrix}
 {T^{ - 1}}   & 0   \\
 0   & 0   \\
\end{matrix} }} \right]U^\ast ,
\\
\label{2-6}
A^{ {\#}}
&
=
U\left[ {{\begin{matrix}
 {T^{ - 1}}   & T^{ - 2}S   \\
 0   & 0   \\
\end{matrix} }} \right]U^\ast,
\end{align}
where
$T={\Sigma K}$
and
$S={\Sigma L}$.
\end{lemma}

It is well known that the core-nilpotent decomposition
has been widely used in matrix theory
\cite{Ben2003book,Liu2016book,Mitra2010book}:
\begin{lemma}
\label{C-N-Decomposition}
{\rm{\cite[Core-nilpotent decomposition]{Mitra2010book}}}
  Let $A\in \mathbb{C}_{n,n}$ be with ${\rm Ind}(A)=k$,
then $A$ can be written as
the sum of matrices $\widehat{A}_1$ and $ {\widehat{A}}_2$,
 i.e.
$A = \widehat{A}_1 + \widehat{A}_2$, where
$$\widehat{A}_1\in { \mathbb{C}}_n^{\mbox{\rm\footnotesize\texttt{CM}}},  \
\widehat{A}^k_2 = 0
 \mbox{\ \rm and\ \ }
 \widehat{A}_1  \widehat{A}_2 =\widehat{ A}_2\widehat{A}_1 = 0.$$
\end{lemma}
Similarly,
Wang introduced the notion of
the core-EP decomposition in \cite{Wang2016laa-A}:
\begin{lemma}
{\rm\cite[Core-EP Decomposition]{Wang2016laa-A} }
\label{core-EP-Decomposition}
Let $A\in \mathbb{C}_{n,n}$ be with ${\rm Ind}(A)=k$,
then $A$ can be written as the sum of matrices $A_1$ and $A_2$,
 i.e.
$A = A_1 + A_2$,
 where

{\rm ({i})}\ \   $A_1\in { \mathbb{C}}_n^{\mbox{\rm\footnotesize\texttt{CM}}}$;

{\rm ({ii})}\ \     $A^k_2 = 0 $;

{\rm ({iii})}\ \   $A_1^\ast A_2 = A_2A_1 = 0$.

\noindent Here one or both of $A_1$ and $A_2$ can be null.
\end{lemma}

\begin{lemma}
{\rm\cite{Wang2016laa-A}}
\label{Theorem-2-2}
 Let
 the core-EP decomposition of $A\in \mathbb{C}_{n,n}$
 be as in Lemma \ref{core-EP-Decomposition}.
 Then there exists a  unitary matrix $U$
 such that
\begin{align}
\label{2-4}
A_1 = U\left[ {{\begin{matrix}
 T   & S   \\
 0   & 0   \\
\end{matrix} }} \right]U^\ast,
\
%{\rm and}
\
A_2 = U\left[ {{\begin{matrix}
 0   & 0   \\
 0   & N   \\
\end{matrix} }} \right]U^\ast  ,
\end{align}
where $T$ is   non-singular, and $N$ is  nilpotent.
Furthermore,
the core-EP inverse of $A $ is
  \begin{align}
 \label{2-5}
 A^{\tiny\textcircled{\dag}}
 =
 U\left[ {{\begin{matrix}
 {T^{ - 1}}   & 0   \\
 0   & 0   \\
\end{matrix} }} \right]U^\ast.
\end{align}
\end{lemma}

\section{ WG inverse}
\label{Sec-new-GI} 
In this section,
we apply the core-EP decomposition to introduce a generalized group inverse (i.e. the WG inverse)
and
consider some characterizations of   the generalized inverse.

\subsection{  Definition and properties of the WG inverse}

Let  $A\in \mathbb{C}_{n,n}$
be with  ${\rm Ind}(A)=k$,
and
consider the system of equations
 \begin{align}
 \label{3-1}%
% \left( {1^l} \right)~  XA^{k+1} = A^k,\ \
\left( 2' \right)~ A X^{2} = X, \ \
\left( 3^c \right)~  {AX}    = A^{\tiny\textcircled{\dag}} A.
\end{align}
Let the core-EP decomposition  of $A$ be as in  (\ref{2-4}).
Then
the core-EP inverse $A^{\tiny\textcircled{\dag}}$
of $A$ can be formed as:
\begin{align}
\label{3-2}
A^{\tiny\textcircled{\dag}}
=
U\left[ {{\begin{matrix}
T^{-1}   & 0  \\
0   & 0   \\
\end{matrix} }} \right]U^\ast.
\end{align}
Suppose that
\begin{align}
\label{3-4}
X=
 U\left[ {{\begin{matrix}
T^{-1}    &T^{-2}S    \\
0         & 0  \\
\end{matrix} }} \right]U^\ast.
\end{align}
Substituting
 (\ref{3-4}) for $X$ in  (\ref{3-1})
 and
applying
(\ref{3-2} ),
 we derive
\begin{align*}
 A X^{2} - X
 &
 =U\left[ {{\begin{matrix}
 T    & S   \\
 0   & N   \\
\end{matrix} }} \right]  \left[ {{\begin{matrix}
T^{-2}    &T^{-3}S     \\
0         & 0  \\
\end{matrix} }} \right]U^\ast-U\left[ {{\begin{matrix}
T^{-1}    &T^{-2}S   \\
0         & 0  \\
\end{matrix} }} \right]U^\ast=0;
 \\
 {AX} - A^{\tiny\textcircled{\dag}} A
 &
 =  U\left[ {{\begin{matrix}
 T    & S   \\
 0   & N   \\
\end{matrix} }} \right] \left[ {{\begin{matrix}
T^{-1}    &T^{-2}S    \\
0         & 0  \\
\end{matrix} }} \right]U^\ast
-
U\left[ {{\begin{matrix}
 T^{-1}   & 0  \\
 0   & 0   \\
\end{matrix} }} \right]  \left[ {{\begin{matrix}
 T    & S   \\
 0   & N   \\
\end{matrix} }} \right]U^\ast=0.
\end{align*}
Therefore,
(\ref{3-4})
is the solution of the system to equations (\ref{3-1}).

Furthermore,
suppose that both
$X$
and
$\mathcal{X}$
  satisfy  (\ref{3-1}),
then
  \begin{align}
\nonumber
X=AX^2= A^{\tiny\textcircled{\dag}} AX=A^{\tiny\textcircled{\dag}}A^{\tiny\textcircled{\dag}}A
=A^{\tiny\textcircled{\dag}}A\mathcal{X}
=A\mathcal{X}^2
=\mathcal{X},
\end{align}
that is,
the solution to the system of equations (\ref{3-1}) is unique.
We have the following:
\begin{theorem}
\label{Th-3-1}
The system of equations (\ref{3-1})  is consistent
and
has a unique solution (\ref{3-4}).
\end{theorem}

\begin{definition}
Let $A\in \mathbb{C}_{n,n}$   be a matrix of index $k$.
 %(not necessarily $\leq$ 1).
The WG inverse of $A$,
denoted as $  A^{\tiny\textcircled{W}}$,
 is defined to be
the solution to
the system (\ref{3-1}) .
\end{definition}

\begin{remark}
\label{Remark-3-3}
 When  $A\in \mathbb{C}^{\mbox{\rm\footnotesize\texttt{CM}}}_{n }$,
we have
 $  A^{\tiny\textcircled{W}}=A^\#$.
\end{remark}

\begin{remark}
\label{Remark-3-7}
In  \cite[Definition 1]{Campbel1978llaa},
the notion of weak Drazin inverse was given:
let $A\in \mathbb{C}_{n, n}$ and  ${\rm Ind}(A)=k$,
then
$X$ is  a  weak Drazin inverse of $A$ if $X$ satisfies
(1$^k$).
Applying (\ref{3-4}),
it is easy to check that the WG inverse $A^{\tiny\textcircled{W}}$ is a weak Drazin inverse of $A$.
\end{remark}

More details about the weak Drazin inverse
can be seen
in \cite{Campbel1978llaa,Campbell2009book,Wang2016laa}.

\bigskip

In the following example,
we explain that the WG inverse is
different from
the Drazin, DMP, core-EP
and
 B-T inverses.

\begin{example}
\label{Ex3-1}
Let
$A=\left[ {{\begin{matrix}
 1  & 0  & 1 & 0 \\
 0  & 1 & 0 & 1  \\
 0  & 0 & 0 & 1   \\
 0  & 0 & 0 & 0   \\
\end{matrix} }} \right]$.
It is easy to check that ${\rm Ind}(A) = 2$,
the Moore-Penrose inverse   $A^{\dag}$
and
the Drazin inverse $A^D$
are
$$A^{\dag}=\left[ {{\begin{matrix}
 0.5  & 0  & 0 & 0 \\
 0  & 1 & -1 & 0   \\
 0.5 & 0 & 0 & 0   \\
 0  & 0 & 1 & 0   \\
\end{matrix} }} \right] \ {\mbox{and }} \
A^{D}=\left[ {{\begin{matrix}
     1   &  0   &  1 &    1\\
     0   &  1   &  0  &   1\\
     0   &  0   &  0 &    0\\
     0   &  0   &  0  &   0 \\
\end{matrix} }} \right],$$
the DMP inverse   $A^{d, \dag}$
and
the B-T inverse $A^\diamondsuit$
are
$$A^{d,\dag}=A^{D}AA^{\dag}=\left[ {{\begin{matrix}
  1     &     0     &   1    &    0\\
  0     &     1     &   0    &    0\\
  0     &     0     &   0    &    0\\
  0     &     0     &    0   &    0\\
\end{matrix} }} \right] \ {\mbox{and }} \
A^{\diamondsuit}
=\left(A^2A^{\dag}\right)^\dag
=\left[ {{\begin{matrix}
         0.5   &    0     &    0   &      0\\
         0     &    1     &    0   &      0\\
         0.5   &    0     &    0   &      0\\
         0     &    0     &    0    &     0\\
\end{matrix} }} \right],$$
and
the core-EP inverse $A^{\tiny\textcircled{\dag}}$
and
the WG inverse $A^{\tiny\textcircled{W}} $
are
$$
A^{\tiny\textcircled{\dag}}=\left[ {{\begin{matrix}
     1   &  0   &  0 &    0\\
     0   &  1   &  0  &   0\\
     0   &  0   &  0 &    0\\
     0   &  0   &  0  &   0 \\
\end{matrix} }} \right]
\ {\mbox{and }} \
A^{\tiny\textcircled{W}}=\left[ {{\begin{matrix}
  1     &     0     &   1    &    0\\
  0     &     1     &   0    &    1\\
  0     &     0     &   0    &    0\\
  0     &     0     &    0   &    0\\
\end{matrix} }} \right].$$
\end{example}

\bigskip

\subsection{Characterizations of the WG inverse}
   Let $A = \widehat{A}_1 +\widehat{ A}_2 $
   be the   core-nilpotent  decomposition  of
   $A\in \mathbb{C}_{n,n}$.
 Then
 $A^D=\widehat{A}_1^\#$.
Applying   Lemma \ref{core-EP-Decomposition},
 (\ref{2-4})
and (\ref{3-4}),
we have the following theorem.
\begin{theorem}
\label{Remark-3-2}%
%   Let $A = {A}_1 +  {A}_2 $ be the   core-EP  decomposition  of $A\in \mathbb{C}_{n,n}$.
 Let
 the core-EP decomposition of $A\in \mathbb{C}_{n,n}$
 be as in (\ref{2-4}).
Then
\begin{align}
\label{3-8}
A^{\tiny\textcircled{W}}= {A}_1^{\#}=
 U\left[ {{\begin{matrix}
T^{-1}    &T^{-2}S    \\
0         & 0  \\
\end{matrix} }} \right]U^\ast.
\end{align}
\end{theorem}

\bigskip

Since
 \begin{align*}
 A A^{\tiny\textcircled{\dag}}A
 &
 =
 U\left[ {{\begin{matrix}
 T    & S   \\
 0   & N   \\
\end{matrix} }} \right]\left[ {{\begin{matrix}
 T^{-1}    & 0  \\
 0   & 0   \\
\end{matrix} }} \right]\left[ {{\begin{matrix}
 T    & S   \\
 0   & N   \\
\end{matrix} }} \right]U^\ast
\\
 &
 =
 U\left[ {{\begin{matrix}
 I & 0 \\
 0   & 0  \\
\end{matrix} }} \right]\left[ {{\begin{matrix}
 T    & S   \\
 0   & N   \\
\end{matrix} }} \right]U^\ast
 =
 U\left[ {{\begin{matrix}
 T    & S   \\
 0    & 0   \\
\end{matrix} }} \right] U^\ast,
 \end{align*}
 and
 \begin{align*}
\left(A^{\tiny\textcircled{\dag}}\right)^2
 &
 =\left(A^2\right) ^{\tiny\textcircled{\dag}}
 =U\left[ {{\begin{matrix}
 T^{-2}    & 0    \\
 0         & 0   \\
\end{matrix} }} \right]U^\ast,
 \end{align*}
we have the following theorem:
\begin{theorem}
Let $A\in \mathbb{C}_{n,n}$ be with ${\rm Ind}(A)=k$.
Then
  \begin{align*}
 A^{\tiny\textcircled{W} }
  =\left(AA^{\tiny\textcircled{\dag}}A\right)^ {\#}
  =\left(A^{\tiny\textcircled{\dag}}\right)^2A
    =\left(A^2\right)^{\tiny\textcircled{\dag}}A.
 \end{align*}
\end{theorem}

Let the core-EP decomposition  of $A$ be as in  (\ref{2-4}).
Then
\begin{align}
\label{3-3}
 A^k=U\left[ {{\begin{matrix}
 T^k   & \Phi   \\
 0   & 0   \\
\end{matrix} }} \right]U^\ast, \
 A^{k+1}=U\left[ {{\begin{matrix}
 T^{k+1}   & T\Phi  \\
 0   & 0   \\
\end{matrix} }} \right]U^\ast,
\end{align}
where  $\Phi ={\sum\limits_{i = 1}^k {T^{i - 1}SN^{k - i}}}$.
It follows that
\begin{align}
\nonumber
 A^k
 \left(A^{k+2}\right)^{\tiny\textcircled{\#} }A
 &
  =
 U\left[ {{\begin{matrix}
 T^k   & \Phi   \\
 0   & 0   \\
\end{matrix} }} \right]
\left[ {{\begin{matrix}
 T^{-(k+2)}   & 0  \\
 0   & 0   \\
\end{matrix} }} \right]
\left[ {{\begin{matrix}
 T    & S  \\
 0    & N   \\
\end{matrix} }} \right]U^\ast
\\
 &
 =
U\left[ {{\begin{matrix}
 T^{-1}    & T^{-2}S   \\
 0   & 0   \\
\end{matrix} }} \right]U^\ast
\label{3-16}
=
A^{\tiny\textcircled{W}},
\\
\nonumber
 \left( A^{k+2}
 \left(A^k\right)^{\dag}\right)^{\dag}A
 &
  =
 \left( A^2A^{k }
 \left(A^k\right)^{\dag}\right)^{\dag}A
 \\
 &
 =
 U
\left[ {{\begin{matrix}
 T^{2}   & 0  \\
 0   & 0   \\
\end{matrix} }} \right]^{\dag}
\left[ {{\begin{matrix}
 T    & S  \\
 0    & N   \\
\end{matrix} }} \right]U^\ast
\label{3-17}
=
A^{\tiny\textcircled{W} }.
\end{align}
Therefore, we have the following theorem.
\begin{theorem}
Let $A\in \mathbb{C}_{n,n}$ be with ${\rm Ind}(A)=k$.
Then
  \begin{align*}
 A^{\tiny\textcircled{W} }
  =A^k
 \left(A^{k+2}\right)^{\tiny\textcircled{\#} }A
  =\left( A^{ 2} P_{A^k}\right)^{\dag}A.
 \end{align*}
\end{theorem}

\bigskip

It is known that
the Drazin inverse
 is  one  generalization of the group inverse.
We will see
 the similarities  and differences
 between the Drazin inverse and the WG inverse
 from
 the following  corollaries.

\begin{corollary}
\label{Remark-3-1}
Let $A\in \mathbb{C}_{n,n}$   be with ${\rm Ind}(A)=k$.
Then
 $$
 {\rm rk}\left(A^{\tiny\textcircled{W}}\right)
 =
 {\rm rk}\left(A^D\right)
 =
 {\rm rk}\left(A^k\right).
 $$
\end{corollary}

It is well known that
 $\left(A^2\right)^{D}
=  \left(A^{D}\right)^2$,
but
the same is not true for the WG inverse.
Applying the core-EP decomposition (\ref{2-4}) of  $A$,
we have
\begin{align}
\label{3-9}
A^2
=
U\left[ {{\begin{matrix}
 T^2    & TS +SN  \\
 0      & N^2   \\
\end{matrix} }} \right] U^\ast
\end{align}
and
\begin{align}
\label{3-10}
\left(A^2\right)^{\tiny\textcircled{W}}
=
U\left[ {{\begin{matrix}
 T^{-2}    &  T^{-4}\left(TS +SN\right)  \\
 0         & 0   \\
\end{matrix} }} \right] U^\ast,\ \
\left(A^{\tiny\textcircled{W}}\right)^2
=
 U\left[ {{\begin{matrix}
T^{-2}    &T^{-3}S    \\
0         & 0  \\
\end{matrix} }} \right]U^\ast.
\end{align}
Therefore,
$
\left(A^2\right)^{\tiny\textcircled{W}}
=
\left(A^{\tiny\textcircled{W}}\right)^2 $
if and only if
$
T^{-4}\left(TS +SN\right)
=
T^{-3}S $.
Since
$T$ is invertible,
we derive the following Corollary \ref{Remark-3-5}.
\begin{corollary}
\label{Remark-3-5}
Let
 the core-EP decomposition of $A\in \mathbb{C}_{n,n}$
 be as in (\ref{2-4}).
Then
$\left(A^2\right)^{\tiny\textcircled{W}}
=
\left(A^{\tiny\textcircled{W}}\right)^2$
if and only if
$SN=0$.
\end{corollary}

The  commutativity is one of the main characteristics of the group inverse.
The Drazin inverse has the characteristic, too.
It is of interest to inquire whether
the same is  true  or not for the WG inverse.
Applying the core-EP decomposition (\ref{2-4}) of  $A$,
we have
\begin{subequations}
\begin{align}
\label{3-13}
A A^{\tiny\textcircled{W}}
&
=
U\left[ {{\begin{matrix}
 T    & S   \\
 0   & N   \\
\end{matrix} }} \right]\left[ {{\begin{matrix}
 T^{-1}    & T^{-2}S  \\
 0   & 0   \\
\end{matrix} }} \right]U^\ast
 =
 U\left[ {{\begin{matrix}
 I  & T^{-1}S  \\
 0   & 0  \\
\end{matrix} }} \right]U^\ast;
\\
\label{3-14}
 A^{\tiny\textcircled{W}}  A
 &
=
U\left[ {{\begin{matrix}
 T^{-1}    & T^{-2}S  \\
 0   & 0   \\
\end{matrix} }} \right]
\left[ {{\begin{matrix}
 T    & S   \\
 0   & N   \\
\end{matrix} }} \right]U^\ast
 =
 U\left[ {{\begin{matrix}
 I & T^{-1}S +T^{-2}SN \\
 0   & 0  \\
\end{matrix} }} \right]U^\ast.
 \end{align} \end{subequations}
Therefore,
we have the following Corollary \ref{Remark-3-6}.
\begin{corollary}
\label{Remark-3-6}
%Let $A$   be with ${\rm Ind}(A)=k$,
%and
 Let
 the core-EP decomposition of $A\in \mathbb{C}_{n,n}$
 be as in (\ref{2-4}).
Then
$A A^{\tiny\textcircled{W}}  = A^{\tiny\textcircled{W}}  A$
if and only if
$SN=0$.
\end{corollary}

\bigskip

Let $SN=0$,
then by applying Corollary \ref{Remark-3-5}
and
Corollary \ref{Remark-3-6},
we derive
\begin{align}
\nonumber
A^2
&
=
U\left[ {{\begin{matrix}
 T^2    & TS   \\
 0      & N^2   \\
\end{matrix} }} \right] U^\ast,
...
,
A^{k }
=
U\left[ {{\begin{matrix}
 T^{k }    & T^{k-1 }S   \\
 0      & 0   \\
\end{matrix} }} \right] U^\ast, \
A^{k+1}
=
U\left[ {{\begin{matrix}
 T^{k+1}    & T^{k }S   \\
 0      & 0   \\
\end{matrix} }} \right] U^\ast.
\end{align}
Let $t$ be a positive integer. It follows from applying
 (\ref{2-7}),
 (\ref{2-6})
 and
 (\ref{2-5})
 that
\begin{align}
\nonumber
\left(A^{t+1}\right)^{\tiny\textcircled{\dag}}
&
=
U
\left[ {{\begin{matrix}
 T^{-(t+1)}    &0   \\
 0      & 0   \\
\end{matrix} }} \right]
U^\ast,
\\
\nonumber
\left(A ^{k+1}\right)^{\#}
&
=
\left(A^{k+1}\right)^{\tiny\textcircled{\#}}
=
U
\left[ {{\begin{matrix}
 T^{-(k+1)}    &T^{-(k+2)}S   \\
 0      & 0   \\
\end{matrix} }} \right]
U^\ast,
\\
\nonumber
A^D
&
=
\left(A^{k+1}\right)^{\#}A^{k }
=
U
\left[ {{\begin{matrix}
 T^{-(k+1)}    &T^{-(k+2)}S   \\
 0      & 0   \\
\end{matrix} }} \right]\left[ {{\begin{matrix}
 T^{k }    & T^{k-1 }S   \\
 0      & 0   \\
\end{matrix} }} \right]
U^\ast
\\
\nonumber
&
=
U
\left[ {{\begin{matrix}
 T^{-1}    &T^{-2}S   \\
 0      & 0   \\
\end{matrix} }} \right]
U^\ast
=
A^{\tiny\textcircled{W}}.
\end{align}

Therefore, we have the following Corollary \ref{Remark-3-10}..
\begin{corollary}
\label{Remark-3-10}
Let $A\in \mathbb{C}_{n,n}$   be with ${\rm Ind}(A)=k$,
 the core-EP decomposition of $A$  be as in (\ref{2-4})
and $SN=0$.
Then
$$
A^{\tiny\textcircled{W}}
=
A ^{D}
=
\left(A^{k+1}\right)^{\tiny\textcircled{\#}}A^k
=
\left(A^{t+1}\right)^{\tiny\textcircled{\dag}}A^t,$$
where $t$ is a  positive integer.
\end{corollary}

\section{Two Orders} %: the WG order and the core-group-minus partial order
\label{Sec-new-PO}

A binary operation
on a set $S$ is said to be a  \emph{pre-order} on $S$
 if
 it is
reflexive
and
 transitive.
 If the pre-order is also anti-symmetric,
 we call it a \emph{partial order}
\cite[Chap 1]{Mitra2010book}.
 Let
 $S_1$ and $S_2$ be sets,
 and
 $S_2\subseteq S_1$,
then  a  partial order $\mathop \leq \limits^ { { {1}}}$
  on   $S_1$
 is said to be \emph{implied} by
 a  partial order $\mathop \leq \limits^ { { {2}}}$
  on  $S_2$
 if for  $A,B\in S_2$,
$$A \mathop \leq \limits^ { { {2}}} B
\Rightarrow
A \mathop \leq \limits^ { { {1}}} B.$$
 The expression
 $ A \mathop \nleq \limits^ { { {2}}} B$
 means
 that
 $A$ is not below  $B$ under
 the partial order $\mathop \leq \limits^ { { {2}}}$.

In \cite[Definition 4.4.1 \mbox{and} Definition 4.4.17]{Mitra2010book},
the definitions  of the \emph{Drazin order} and the \emph{C-N partial order} are given:
  \begin{align}
\label{5-12}
 A \mathop \le \limits^ {D}B
 &
:
 A, B \in {\mathbb{C} }_{n,n},
 \widehat{A}_1 \mathop \le \limits^ {\#} \widehat{B}_1  ,
\\
\label{5-13}
 A \mathop \le \limits^ {\#, -}B
 &
   :
 A, B \in {\mathbb{C} }_{n,n},
 \widehat{A}_1 \mathop \le \limits^ {\#} \widehat{B}_1
{\mbox{\ \rm and\ }}
 \widehat{A}_2  \mathop \le \limits^ -   {\widehat{B}_2}  ,
\end{align}
  in which $  A = \widehat{A}_1 + \widehat{A}_2$
  and
$B = \widehat{B}_1 + \widehat{B}_2$
 are the core-nilpotent decompositions of  $A$    and $ B$,
 respectively.
 Similarly,
 in this section,
 we apply the core-EP decomposition
 to introduce  two orders:
one is the WG order and the other is the C-E order.

\subsection{WG  order}
Consider
  the  binary operation:
  \begin{align}
\label{5-10}
 A \mathop \le \limits^ {\mbox{\tiny{\rm WG}}}B
 &
:
 A,B \in {\mathbb{C} }_{n,n} ,\
 A_1 \mathop \le \limits^ {\#} B_1  ,
\end{align}
  in which
  $A = A_1 + A_2$
  and
  $B = B_1 + B_2$
 are the core-EP decompositions of   $A$    and  $ B$, respectively.

Reflexivity of the relation is obvious.
Suppose
   $A \mathop \le \limits^{\mbox{\tiny\rm{WG}}} B $
   and
   $B \mathop \le \limits^{\mbox{\tiny\rm{WG}}} C $,
in which
   $A = A_1 + A_2 $,
   $B = B_1 + B_2 $
   and
   $C = C_1 + C_2 $
are the core-EP decompositions of $A$,  $B$
and
$C$, respectively.
Then
$A_1 \mathop \le \limits^ {\#} B_1 $
and
$B_1 \mathop \le \limits^ {\#} C_1 $.
Therefore
$A_1 \mathop \le \limits^ {\#} C_1 $.
It follows from (\ref{5-10}) that
   $A \mathop \le \limits^{\mbox{\tiny\rm{WG}}} C$.
\begin{example}
\label{Ex.5-1}
Let
$$
A=\left[
          \begin{array}{ccc}
            1 & 1 & 1 \\
            0 & 0 & 1 \\
            0 & 0 & 0\\
          \end{array}
\right],
\
B=\left[
          \begin{array}{ccc}
            1 & 1 & 1 \\
            0 & 0 & 2 \\
            0 & 0 & 0\\
          \end{array}
\right]
$$
Although
$ A \mathop \le \limits^ {\mbox{\tiny{\rm WG}}}B$
and
$ B \mathop \le \limits^ {\mbox{\tiny{\rm WG}}}A$,
$A\neq B$.
Therefore,
the  anti-symmetry
of the  binary operation (\ref{5-10})
 cannot be tenable.
\end{example}

Therefore,
we have the following Theorem \ref{Def-5-1}.
\begin{theorem}
\label{Def-5-1}
The   binary operation (\ref{5-10})
is a pre-order.
We call this pre-order the WG order.
\end{theorem}

\begin{remark}
\label{Remark-5-2}
%When $A,B\in { \mathbb{C}}_n^{\mbox{\rm\footnotesize\texttt{CM}}}$,
The WG order coincides with  the sharp partial order
on ${\mathbb{C}}_n^{\mbox{\rm\footnotesize\texttt{CM}}}$.
\end{remark}

In the following two examples,
we see some differences between the WG order and the Drazin order.

\begin{example}
\label{Ex.5-2}
Let
$$
A=\left[
          \begin{array}{ccc}
            1 & 1 & 1 \\
            0 & 0 & 1 \\
            0 & 0 & 0\\
          \end{array}
\right],
\
B=\left[
          \begin{array}{ccc}
            1 & 1 & 1 \\
            0 & 0 & 2 \\
            0 & 0 & 0\\
          \end{array}
\right],\
A^D=\left[
          \begin{array}{ccc}
            1 & 1 & 2 \\
            0 & 0 & 0 \\
            0 & 0 & 0\\
          \end{array}
\right].
$$
It is easy to check that
$ A \mathop \le \limits^ {\mbox{\tiny{\rm WG}}}B$.

Since
$A^DA\neq A^D B$,
we  derive $ A \mathop \nleq \limits^ { { {  D}}} B$.
Therefore,
the WG order
does not imply
the Drazin order.
\end{example}

\begin{example}
\label{Ex.5-3}
Let
\begin{align*}
\widehat{A}
&
= \left[
          \begin{array}{ccc}
            1 & 0 & 0 \\
            0 & 0 & 0 \\
            0 & 0 & 0\\
          \end{array}
\right],
\widehat{B}=\left[
          \begin{array}{ccc}
            1 & 0 & 0 \\
            0 & 0 & 1 \\
            0 & 0 & 0\\
          \end{array}
\right],
P=
\left[
          \begin{array}{ccc}
            1 & -2 & 0 \\
            0 & 1 & 0 \\
            0 & 0 & 1\\
          \end{array}
\right],
\\
A
&
=P\widehat{A}P^{-1}
=\left[
          \begin{array}{ccc}
            1 & 2 & 0 \\
            0 & 0 & 0 \\
            0 & 0 & 0\\
          \end{array}
\right],
B
=P\widehat{B}P^{-1}
=\left[
          \begin{array}{ccc}
            1 & 2 & -2 \\
            0 & 0 & 1 \\
            0 & 0 & 0\\
          \end{array}
\right],
\\
A_1
&
=
\left[
          \begin{array}{ccc}
            1 & 2 & 0 \\
            0 & 0 & 0 \\
            0 & 0 & 0\\
          \end{array}
\right],\
A_2=0,\
B_1
=\left[
          \begin{array}{ccc}
            1 & 2 & -2 \\
            0 & 0 & 0 \\
            0 & 0 & 0\\
          \end{array}
\right],\
B_2
=\left[
          \begin{array}{ccc}
            0 & 0 &0 \\
            0 & 0 & 1 \\
            0 & 0 & 0\\
          \end{array}
\right],
\end{align*}
in which
   $A = A_1 + A_2 $
   and
   $B = B_1 + B_2 $
are the core-EP decompositions of $A$ and $B$, respectively.
Then
$ A \mathop \leq \limits^ { { {  D}}} B$
and
$A_1 \mathop \nleq \limits^ { \# } B_1$.
Therefore,
the Drazin order
does not imply
the WG  order.
\end{example}

\bigskip

It is well known that
$A \mathop \leq \limits^ { { {  D}}} B
\Rightarrow
A^2 \mathop \leq \limits^ { { {  D}}} B^2$,
but the same is not true
for the WG order
as the following example shows:
\begin{example}
\label{Remark-5-4}
Let
$$
A=\left[
          \begin{array}{ccc}
            1 & 1 & 1 \\
            0 & 0 & 0 \\
            0 & 0 & 0\\
          \end{array}
\right],
\
B=\left[
          \begin{array}{ccc}
            1 & 1 & 1 \\
            0 & 0 & 2 \\
            0 & 0 & 0\\
          \end{array}
\right],\
A^2=\left[
          \begin{array}{ccc}
            1 & 1 & 1 \\
            0 & 0 & 0 \\
            0 & 0 & 0\\
          \end{array}
\right],\
B^2=\left[
          \begin{array}{ccc}
            1 & 1 & 3 \\
            0 & 0 & 0 \\
            0 & 0 & 0\\
          \end{array}
\right].
$$
We  derive
$ A^2 \mathop \nleq \limits^ {\mbox{\tiny{\rm WG}}}B^2 $.
Therefore,
$ A \mathop \leq \limits^ {\mbox{\tiny{\rm WG}}}B
\nRightarrow
A^2 \mathop \leq \limits^ {\mbox{\tiny{\rm WG}}}B^2$. 
\end{example}

Let
$ A \mathop \le \limits^ {\mbox{\tiny{\rm WG}}}B$,
$A = A_1 + A_2$
  and
  $B = B_1 + B_2$
 are the core-EP decompositions of   $A$    and  $ B$,
  $A_1$ and $A_2$  be as given in (\ref{2-4}),
 and
% denote the
  partition %ing of
% $U^\ast B_1 U$ as
 \begin{align}
\label{5-3}
U^\ast B_1 U
&
=
\left[ {{\begin{matrix}
 B_{11}    & B_{12}   \\
 B_{21}    & B_{22}   \\
\end{matrix} }} \right].
 \end{align}
Applying
(\ref{3-13})
and
(\ref{3-14})
gives
\begin{align}
\nonumber
A_1 A_1^\#
&
=
U\left[ {{\begin{matrix}
 T   &  S   \\
 0   &  0   \\
\end{matrix} }} \right]\left[ {{\begin{matrix}
 T^{-1}    & T^{-2}S  \\
 0         & 0   \\
\end{matrix} }} \right]U^\ast
 =
 U\left[ {{\begin{matrix}
 I  & T^{-1}S  \\
 0   & 0  \\
\end{matrix} }} \right]U^\ast;
\\
\nonumber
B_1 A_1^\#
&
=
U\left[ {{\begin{matrix}
 B_{11}    & B_{12}   \\
 B_{21}    & B_{22}   \\
\end{matrix} }} \right]\left[ {{\begin{matrix}
 T^{-1}    & T^{-2}S  \\
 0   & 0   \\
\end{matrix} }} \right]U^\ast
 =
 U\left[ {{\begin{matrix}
 B_{11}T^{-1}    &B_{11} T^{-2}S  \\
 B_{21}T^{-1}    &B_{21} T^{-2}S  \\
\end{matrix} }} \right]U^\ast.
\end{align}
Since
$A \mathop \leq \limits^ {\mbox{\tiny{\rm WG}}}B$,
$A_1 \mathop \leq \limits^ { \# }B_1$.
It follows
from
$ A_1A_1^\#
=
B_1 A_1^\#$
that
\begin{align}
\label{5-1}
B_{11}=T \mbox{\ and \ } B_{21}=0.
 \end{align}
By applying (\ref{5-3}) and (\ref{5-1}),
we have
\begin{align}
\nonumber
 A_1^\#  A_1
 &
=
 U\left[ {{\begin{matrix}
 I   & T^{-1}S   \\
 0   & 0  \\
\end{matrix} }} \right]U^\ast,
\\
\nonumber
 A_1^\#B_1
&
 =
 U\left[ {{\begin{matrix}
  I    &  T^{-1}B_{12} +T^{-2}SB_{22} \\
  0    & 0 \\
\end{matrix} }} \right]U^\ast.
 \end{align}
It follows
from
$ A_1^\#  A_1
=
A_1^\#B_1$
that
\begin{align}
\nonumber
0
&
=
T^{-1}S -\left(T^{-1}B_{12} +T^{-2}SB_{22}\right)\\
\nonumber
 &
 =
 T^{-1}\left( S -T^{-1}SB_{22}- B_{12} \right).
 \end{align}
Therefore,
\begin{align}
\label{5-2}
B_{12}= S -T^{-1}SB_{22},
 \end{align}
 in which
$B_{22}$ is an arbitrary matrix of an appropriate size. 
From (\ref{5-1}) and (\ref{5-2}),
we obtain
\begin{align}
\label{5-4}
B_1&
=
U\left[ {{\begin{matrix}
T   & S -T^{-1}SB_{22}  \\
0   & B_{22}   \\
\end{matrix} }} \right]U^\ast.
 \end{align}

Since $B_1$ is core invertible,
 $B_{22}$ is core invertible.
Let the core decomposition of $B_{22}$ be as
 \begin{align}
\label{5-5}
B_{22}&
=
U_1\left[ {{\begin{matrix}
T_1  & S_1  \\
0   & 0   \\
\end{matrix} }} \right]U_1^\ast,
 \end{align}
where $T_1$ is invertible.
Denote
$$\widehat{U}=U\left[ {{\begin{matrix}
I   & 0  \\
0   & U_1 \end{matrix}}}\right].$$
It is easy to see that
$\widehat{U}$
is a unitary matrix.
Let $SU_1$ be partitioned
as follows:
$$SU_1=\left[ {{\begin{matrix}
\widehat{S}_1  & \widehat{S}_2   \end{matrix}}}\right].$$
 Then
\begin{align}
\label{5-7}
A_1
=
&
\widehat{U}
\left[ {{\begin{matrix}
T   &    \widehat{S}_1   &  \widehat{S}_2   \\
0   &  0  & 0  \\
0   & 0  & 0   \\
\end{matrix} }} \right]\widehat{U}^\ast
 \end{align}
 and
\begin{align}
\nonumber
B_1
&
=
U\left[ {{\begin{matrix}
T   & S -T^{-1}SB_{22}  \\
0   & {U_1\left[ {{\begin{matrix}
T_1  & S_1  \\
0   & 0   \\
\end{matrix} }} \right]U_1^\ast }   \\
\end{matrix} }} \right]U^\ast
\\
\nonumber
&
=
U\left[ {{\begin{matrix}
I   & 0  \\
0   & U_1 \end{matrix}}}\right]
\left[ {{\begin{matrix}
T   & SU_1 -T^{-1}SU_1U_1^\ast B_{22}U_1  \\
0   & {\left[ {{\begin{matrix}
T_1  & S_1  \\
0   & 0   \\
\end{matrix} }} \right] }   \\
\end{matrix} }} \right]\left[ {{\begin{matrix}
I   & 0  \\
0   & U_1^\ast \end{matrix}}}\right]U^\ast
\\
\nonumber
&
=
\widehat{U}
\left[ {{\begin{matrix}
T   & \left[ {{\begin{matrix}
\widehat{S}_1  & \widehat{S}_2   \end{matrix}}}\right]
 -
 T^{-1}\left[ {{\begin{matrix}
\widehat{S}_1  & \widehat{S}_2   \end{matrix}}}\right]
{\left[ {{\begin{matrix}
T_1  & S_1  \\
0   & 0   \\
\end{matrix} }} \right] }  \\
0   & {\left[ {{\begin{matrix}
T_1  & S_1  \\
0   & 0   \\
\end{matrix} }} \right] }   \\
\end{matrix} }} \right]\widehat{U}^\ast
\\
\label{5-6}
&
=
\widehat{U}
\left[ {{\begin{matrix}
T   &    \widehat{S}_1 -T^{-1}\widehat{S}_1 T_1  &  \widehat{S}_2 - T^{-1}\widehat{S}_1 S_1 \\
0   &  T_1  & S_1  \\
0   & 0  & 0   \\
\end{matrix} }} \right]\widehat{U}^\ast.
 \end{align}
From (\ref{5-10}),
(\ref{5-7})
and
(\ref{5-6}),
we derive the following Theorem \ref{Theorem-5-1-1}.
\begin{theorem}
\label{Theorem-5-1-1}
Let  $A,B \in {\mathbb{C} }_{n,n}$.
Then
$
A \mathop \le \limits^{\mbox{\tiny{\rm WG}}} B$
if and only if
there exists a unitary matrix
$\widehat{U}$
such that
\begin{subequations}
\begin{align}
\label{5-8-a}
A
&
=
\widehat{U}
\left[ {{\begin{matrix}
T   &    \widehat{S}_1   &  \widehat{S}_2   \\
0   &  N_{11}  & N_{12}   \\
0   & N_{21}   & N_{22}   \\
\end{matrix} }} \right]\widehat{U}^\ast,
\\
\label{5-8-b}
B
&
=
\widehat{U}
\left[ {{\begin{matrix}
T   &    \widehat{S}_1 -T^{-1}\widehat{S}_1 T_1  &  \widehat{S}_2 - T^{-1}\widehat{S}_1 S_1 \\
0   &  T_1  & S_1  \\
0   & 0  & N_2   \\
\end{matrix} }} \right]\widehat{U}^\ast,
 \end{align}
 \end{subequations}
where
$T$ and $T_1$
are
invertible,
 $\left[ {{\begin{matrix}
N_{11}  & N_{12}   \\
  N_{21}   & N_{22}   \\
\end{matrix} }} \right]$
and
$N_2$ are nilpotent.
\end{theorem}

\subsection{C-E partial order}
Consider
  the  binary operation:
\begin{align}
\label{5-11}
 A\mathop \le \limits^{\rm CE} B
:
 A,B \in {\mathbb{C} }_{n,n} ,
 A_1 \mathop \le \limits^ {\#} B_1
{\mbox{\ \rm and\ }}
 A_2 \mathop \le \limits^ {-} B_2  ,
\end{align}
  in which
  $A = A_1 + A_2$
  and
  $B = B_1 + B_2$
 are the core-EP decompositions of   $A$    and  $ B$, respectively.

\begin{definition}
\label{Def-5-2}
Let $A,B\in \mathbb{C}_{n,n}$.
We say that $A$ is  below $B$ under
 the C-E order
 if
$A$ and $B$ satisfy the binary operation (\ref{5-11}).

When $A$ is below $B$ under the  C-E order,
we write
   $A \mathop \le \limits^{{ \rm  CE}} B $.
\end{definition}

\begin{theorem}
\label{Theorem-5-2}
The C-E order  is a partial order.
\end{theorem}
\begin{proof}
Reflexivity is trivial.

Let
   $A \mathop \le \limits^{{\rm CE}} B $£»
   $B \mathop \le \limits^{{\rm CE}} C $
and
   $A = A_1 + A_2 $,
   $B = B_1 + B_2 $
   and
   $C = C_1 + C_2 $
are the core-EP decompositions of  $A$,  $B$
and
$C$, respectively.
Then
$A_1 \mathop \le \limits^ {\#} B_1 $,
$B_1 \mathop \le \limits^ {\#} C_1 $
and
$A_2  \mathop \le \limits^ -   {B_2} $,
$B_2  \mathop \le \limits^ -  {C_2} $.
Therefore
$A_1 \mathop \le \limits^ {\#} C_1 $
and
$A_2  \mathop \le \limits^ -   {C_2} $.
It follows
from  Definition \ref{Def-5-2}
that
   $A \mathop \le \limits^{{\rm CE}} C$.

If
   $A \mathop \le \limits^{{\rm CE}} B $
and
   $B \mathop \le \limits^{{\rm CE}} A $,
Then
$A_1 = B_1 $
and
$A_2 =  {B_2} $,
that is,
$A=B$.
\end{proof}

\begin{theorem}
\label{Theorem-5-1}
Let  $A,B \in {\mathbb{C} }_{n,n}$.
Then
$
A \mathop \le \limits^{{\rm CE}} B $
if and only if
there exists
a unitary matrix
$U$
satisfying
\begin{subequations}
\begin{align}
\label{5-9-a}
A
&
=
\widehat{U}
\left[ {{\begin{matrix}
T   &    \widehat{S}_1   &  \widehat{S}_2   \\
0   & 0  & 0   \\
0   & 0  & N_{22}   \\
\end{matrix} }} \right]\widehat{U}^\ast,
\\
\label{5-9-b}
B
&
=
\widehat{U}
\left[ {{\begin{matrix}
T   &    \widehat{S}_1 -T^{-1}\widehat{S}_1 T_1  &  \widehat{S}_2 - T^{-1}\widehat{S}_1 S_1 \\
0   &  T_1  & S_1  \\
0   & 0  & N_2   \\
\end{matrix} }} \right]\widehat{U}^\ast,
 \end{align}
 \end{subequations}
where
  $T$ and $T_1$ are invertible,
  $ N_{22}$ and $N_2$ are nilpotent,
and
$N_{22} \mathop \le \limits^ - N_{2}$.
\end{theorem}

\begin{proof}

Let
   $A \mathop \le \limits^{{\rm CE}} B $£»
   $B \mathop \le \limits^{{\rm CE}} C $
and
   $A = A_1 + A_2 $
   and
   $B = B_1 + B_2 $
are the core-EP decompositions of  $A$ and  $B$, respectively.
Then
$A_1 \mathop \le \limits^ {\#} B_1 $
and
$A_2  \mathop \le \limits^ -   {B_2} $.
It follows
from  Theorem \ref{Theorem-5-1-1}
and
$A_2  \mathop \le \limits^ -  {B_2} $
that
    $\left[ {{\begin{matrix}
N_{11}  & N_{12}   \\
  N_{21}   & N_{22}   \\
\end{matrix} }} \right] \mathop \le \limits^ -  \left[ {{\begin{matrix}
0   & 0  \\
0   & N_{2}   \\
\end{matrix} }} \right]$.
Since
\begin{align}
\nonumber
{\rm rk}\left(N_{22}\right)
\leq
{\rm rk}\left(\left[ {{\begin{matrix}
N_{11}  & N_{12}   \\
  N_{21}   &N_{22}    \\
\end{matrix} }} \right]  \right)
\end{align}
and
\begin{align*}
{\rm rk}\left(N_{2}\right)- {\rm rk}\left(N_{22}\right)
&
\leq
{\rm rk}\left(N_{2}- N_{22}\right)
\leq
{\rm rk}\left(  \left[ {{\begin{matrix}
0   & 0  \\
0   & N_{2}   \\
\end{matrix} }} \right]
-
\left[ {{\begin{matrix}
N_{11}  & N_{12}   \\
  N_{21}   & N_{22}   \\
\end{matrix} }} \right]
\right)
\\
&
=
{\rm rk}\left(N_{2}\right)
-
{\rm rk}\left(\left[ {{\begin{matrix}
N_{11}  & N_{12}   \\
  N_{21}   &N_{22}    \\
\end{matrix} }} \right]  \right),
\end{align*}
we obtain
\begin{align}
\label{2-1-1}
{\rm rk}\left(N_{22}\right)
=
{\rm rk}\left(\left[ {{\begin{matrix}
N_{11}  & N_{12}   \\
  N_{21}   &N_{22}    \\
\end{matrix} }} \right]  \right), \ \
\
N_{22}  \mathop \le \limits^ -  N_{2},
\end{align}
and

\begin{align}
\label{2-1-2}
{\rm rk}\left(N_{2}\right)- {\rm rk}\left(N_{22}\right)
=
{\rm rk}\left(N_{2}\right)
-
{\rm rk}\left(\left[ {{\begin{matrix}
N_{11}  & N_{12}   \\
  N_{21}   &N_{22}    \\
\end{matrix} }} \right]  \right).
\end{align}

Since
$N_{22}  \mathop \le \limits^ -  N_{2}$,
%applying \cite[Theorem 3.7.3]{Mitra2010book}
%gives that
there exist nonsingular
matrices $P$ and $Q$
such that
\begin{align}
\nonumber
 N_{22}=P\left[ {{\begin{matrix}
 D_1 & 0  & 0    \\
  0  & 0  & 0       \\
  0  & 0  & 0       \\
\end{matrix} }} \right]Q,
 \
\
N_2 = P\left[ {{\begin{matrix}
 D_1 & 0  & 0    \\
  0  & D_2  & 0       \\
  0  & 0  & 0       \\
\end{matrix} }} \right]Q,
\end{align}
where $D_1$ and $D_2$ are nonsingular,
(see \cite[Theorem 3.7.3]{Mitra2010book}).
It follows that
\begin{align}
\label{2-1-1-0}
{\rm rk}\left(N_{2}\right)
-
{\rm rk}\left(N_{22}\right)
=
{\rm rk}\left(D_2\right).\end{align}
Denote
\begin{align}\label{Bu-1}
N_{12} =\left[ {{\begin{matrix}
  M_{12}  & M_{13}  & M_{14}   \\
\end{matrix} }} \right]Q,
\ \mbox{\rm and }
\
N_{21} =P\left[ {{\begin{matrix}
  M_{21}  \\ M_{31}  \\ M_{41}   \\
\end{matrix} }} \right].
\end{align}
Then
\begin{align*}
\left[ {{\begin{matrix}
N_{11}     & N_{12}   \\
  N_{21}   &N_{22}    \\
\end{matrix} }} \right]
=
\left[ {{\begin{matrix}
I_{{\rm rk}\left(N_{11}\right)}    & 0  \\
 0                                 &P    \\
\end{matrix} }} \right]
\left[ {{\begin{matrix}
N_{11}  & M_{12}  & M_{13}  & M_{14}   \\
 M_{21}  &D_1 & 0  & 0    \\
 M_{31}  & 0  & 0  & 0       \\
 M_{41}  & 0  & 0  & 0       \\
\end{matrix} }} \right]
\left[ {{\begin{matrix}
I_{{\rm rk}\left(N_{11}\right)}    & 0  \\
 0                                 & Q     \\
\end{matrix} }} \right].
\end{align*}
It follows from (\ref{2-1-1})
that
\begin{align}
\label{Bu-2}
 M_{13}=0,
M_{14}=0,
M_{31}=0
\
\mbox{\rm and \ }M_{41}=0.
\end{align}
Therefore,
\begin{align*}
\left[ {{\begin{matrix}
-N_{11}     & -N_{12}   \\
 - N_{21}   &N_2-N_{22}    \\
\end{matrix} }} \right]
=
\left[ {{\begin{matrix}
I_{{\rm rk}\left(N_{11}\right)}    & 0  \\
 0                                 &P    \\
\end{matrix} }} \right]
\left[ {{\begin{matrix}
-N_{11}  & -M_{12}  &0  & 0   \\
 -M_{21}  &0 & 0  & 0    \\
 0  & 0  & D_2  & 0       \\
 0  & 0  & 0  & 0       \\
\end{matrix} }} \right]
\left[ {{\begin{matrix}
I_{{\rm rk}\left(N_{11}\right)}    & 0  \\
 0                                 & Q     \\
\end{matrix} }} \right].
\end{align*}
By applying
(\ref{2-1-2}),
(\ref{2-1-1-0})
and
$\left[ {{\begin{matrix}
N_{11}  & N_{12}   \\
  N_{21}   & N_{22}   \\
\end{matrix} }} \right] \mathop \le \limits^ -  \left[ {{\begin{matrix}
0   & 0  \\
0   & N_{2}   \\
\end{matrix} }} \right]$ ,
we derive that
\begin{align*}
{\rm rk}\left(\left[ {{\begin{matrix}
0    & 0   \\
0    &N_2     \\
\end{matrix} }} \right]-\left[ {{\begin{matrix}
 N_{11}     &  N_{12}   \\
 N_{21}   & N_{22}    \\
\end{matrix} }} \right]\right)
=&
 {\rm rk}\left(
\left[ {{\begin{matrix}
N_{11}  & M_{12}    \\
M_{21}  &0     \\
\end{matrix} }} \right]\right)
+
 {\rm rk}\left(D_2\right)
 \\
 =&
{\rm rk}\left(N_{2}\right)
-
{\rm rk}\left(N_{22}\right)
\\
=
 &
   {\rm rk}\left(D_2\right).
\end{align*}
Therefore,
$N_{11}=0$,
$M_{12}=0$ and $M_{21}=0$.
By applying (\ref{Bu-1})
and
(\ref{Bu-2}),
we obtain
$N_{11}=0$,
$N_{12}=0$ and $N_{21}=0$.

\bigskip

   Let $A$ and $B$
   be of the forms as given in
   (\ref{5-9-a}) and (\ref{5-9-b}),
then
$A = A_1 + A_2 $
   and
   $B = B_1 + B_2 $
are the core-EP decompositions of  $A$ and  $B$, respectively,
and
\begin{align*}
A_1
&
=
\widehat{U}
\left[ {{\begin{matrix}
T   &    \widehat{S}_1   &  \widehat{S}_2   \\
0   & 0  & 0   \\
0   & 0  & 0   \\
\end{matrix} }} \right]\widehat{U}^\ast, \
A_2=\widehat{U}
\left[ {{\begin{matrix}
0   &   0 &  0  \\
0   & 0  & 0   \\
0   & 0  & N_{22}   \\
\end{matrix} }} \right]\widehat{U}^\ast ;
\\
B_1
&
=
\widehat{U}
\left[ {{\begin{matrix}
T   &    \widehat{S}_1 -T^{-1}\widehat{S}_1 T_1  &  \widehat{S}_2 - T^{-1}\widehat{S}_1 S_1 \\
0   &  T_1  & S_1  \\
0   & 0     & 0  \\
\end{matrix} }} \right]\widehat{U}^\ast,\
B_2
=
\widehat{U}
\left[ {{\begin{matrix}
0   &    0  &  0 \\
0   &  0  & 0  \\
0   & 0  & N_2   \\
\end{matrix} }} \right]\widehat{U}^\ast.
\end{align*}
It is easy to check that $A_1 \mathop \le \limits^ {\#} B_1 $
and
$A_2  \mathop \le \limits^ -   {B_2} $.
Therefore,
$A \mathop \le \limits^{{\rm  CE}} B $.
\end{proof}

\begin{remark}
\label{Remark-3-4}
In Ex. \ref{Ex.5-3},
it is easy to check that
$ A \mathop \le \limits^ {\#,-} B$.
Since
$A_1 \mathop \nleq \limits^ { \# } B_1$,
we have
$A\mathop \nleq \limits^{{\rm CE}}B$.
Therefore,
the C-N partial order
does not imply
the C-E partial order.
\end{remark}

\begin{corollary}
\label{Theorem-5-3}
Let  $A,B \in {\mathbb{C} }_{n,n}$.
If
$A \mathop \le \limits^{{\rm CE}} B $,
then
$A \mathop \le \limits^ -  B$.
\end{corollary}
\begin{proof}
Let  $A,B \in {\mathbb{C} }_{n,n}$.
Then
$A$ and $B$
have the forms as given in Theorem \ref{Theorem-5-1}.
Since $T$ and  $T_1$ are invertible,
%by elementary block matrix operations,
it follows
that
\begin{align*}
{\rm rk}(B )
&
=
{\rm rk}\left(  T \right)
+
{\rm rk}\left(  T_1  \right)
+
{\rm rk}\left(  N_2  \right);
\\
{\rm rk}(A )
&
=
{\rm rk}\left(  T_1 \right)
+
{\rm rk}\left(  N_{22}  \right);
\\
{\rm rk}(B-A)
&
=
{\rm rk}\left(\left[ {{\begin{matrix}
0   &  -T^{-1}\widehat{S}_1 T_1  &   - T^{-1}\widehat{S}_1 S_1 \\
0   &  T_1  & S_1  \\
0   & 0  & N_2 - N_{22}   \\
\end{matrix} }} \right]\right)
\\
&
=
{\rm rk}\left(\left[ {{\begin{matrix}
I_{{\rm rk}\left(  T \right)}  &  T^{-1}\widehat{S}_1  &  0 \\
0   &  I_{{\rm rk}\left(  T_1 \right)}  & 0  \\
0   & 0  & I_{n-{\rm rk}\left(  T \right)-{\rm rk}\left(  T_1 \right)}  \\
\end{matrix} }} \right]
\left[ {{\begin{matrix}
0   &  -T^{-1}\widehat{S}_1 T_1  &   - T^{-1}\widehat{S}_1 S_1 \\
0   &  T_1  & S_1  \\
0   & 0  & N_2 - N_{22}   \\
\end{matrix} }} \right]\right)
\\
&
=
{\rm rk}\left(\left[ {{\begin{matrix}
  T_1  & S_1  \\
  0 & N_2 - N_{22}   \\
\end{matrix} }} \right]\right)
=
{\rm rk}\left(\left[ {{\begin{matrix}
  T_1  & 0  \\
  0    & N_2 - N_{22}   \\
\end{matrix} }} \right]\right)
\\
&
=
{\rm rk}\left(   T_1 \right)
+
{\rm rk}\left(  N_2 - N_{22}  \right)
\\
&
=
{\rm rk}\left(   T_1 \right)
+
{\rm rk}\left(  N_2 \right)
-
{\rm rk}\left(   N_{22}  \right).
\end{align*}
Therefore,
$ {\rm rk}(B -A)={\rm rk}(B )-{\rm rk}(A )$,
that is,
$A \mathop \le \limits^ -  B$.
\end{proof}

\section{ Characterizations of the core-EP order  }
\label{Sec-App-1}
 As is noted in \cite{Wang2016laa-A},
  the core-EP order is given:
\begin{align}
\label{4-4}
A \mathop \le \limits^{\tiny\textcircled{\dag}} B
:
A,B \in {\mathbb{C} }_{n,n},
 A^{\tiny\textcircled{\dag}} A = A^{\tiny\textcircled{\dag}} B
 {\ \rm and\ }
AA^{\tiny\textcircled{\dag}}=BA^{\tiny\textcircled{\dag}}.
\end{align}
 Some characterizations of the core-EP order are given
 in \cite{Wang2016laa-A}.
 \begin{lemma}{\rm \cite{Wang2016laa-A}}
\label{Lemma-1}
Let $A, B\in {\mathbb{C} }_{n,n}$
and
$A \mathop \le \limits^{\tiny\textcircled{\dag}} B$.
Then there exists a  unitary matrix $U$ such that
\begin{align}
\label{4-2}
A = U\left[ {{\begin{matrix}
 {T_1 }   & {T_2 }   & {S_1 }   \\
 0   & {N_{11} }   & {N_{12} }   \\
 0   & {N_{21} }   & {N_{22} }   \\
\end{matrix} }} \right]U^\ast,
\
\
B = U\left[ {{\begin{matrix}
 {T_1 }   & {T_2 }   & {S_1 }   \\
 0   & {T_3 }   & {S_2 }   \\
 0   & 0   & {N_2 }   \\
\end{matrix} }} \right]U^\ast ,
 \end{align}
where
$\left[ {{\begin{matrix}
 {N_{11} }   & {N_{12} }   \\
 {N_{21} }   & {N_{22} }   \\
\end{matrix} }} \right]$ and $N_2 $ are nilpotent,  $T_1$ and $T_3 $ are  non-singular .
\end{lemma}

 Let the core-EP decomposition of $A$ be as given in (\ref{2-4}),
 and denote
 \begin{align}
\label{4-2}
U^\ast B U&
=
\left[ {{\begin{matrix}
 B_1   & B_2  \\
 B_3   & B_4   \\
\end{matrix} }} \right].
 \end{align}
By applying  (\ref{3-13})  and
\begin{align}
\nonumber
B A^{\tiny\textcircled{W}}
&
=
U\left[ {{\begin{matrix}
 B_1   & B_2  \\
 B_3   & B_4   \\
\end{matrix} }} \right]\left[ {{\begin{matrix}
 T^{-1}    & T^{-2}S  \\
 0   & 0   \\
\end{matrix} }} \right]U^\ast
 =
 U\left[ {{\begin{matrix}
 B_1T^{-1}    &B_1 T^{-2}S  \\
 B_3T^{-1}    & B_3T^{-2}S  \\
\end{matrix} }} \right]U^\ast,
\end{align}
we have
$A A^{\tiny\textcircled{W}}=B A^{\tiny\textcircled{W}}$
if and only if
\begin{align}
\nonumber
B_1=T \mbox{\ and \ } B_3=0.
 \end{align}
It follows that
\begin{align}
\nonumber
A^\ast A^{\tiny\textcircled{W}}
 &
=
U\left[ {{\begin{matrix}
 T^\ast    &0  \\
 S^\ast & N ^\ast  \\
\end{matrix} }} \right]
\left[ {{\begin{matrix}
 T^{-1}    & T^{-2}S  \\
 0   & 0   \\
\end{matrix} }} \right]U^\ast
 =
 U\left[ {{\begin{matrix}
 T^\ast T^{-1}    & T^\ast T^{-2}S  \\
 S^\ast T^{-1}    & S^\ast T^{-2}S  \\
\end{matrix} }} \right]U^\ast,
\\
\nonumber
B^\ast A^{\tiny\textcircled{W}}
&
=
U\left[ {{\begin{matrix}
 T ^\ast  & 0 \\
 B_2^\ast & B_4 ^\ast  \\
\end{matrix} }} \right]\left[ {{\begin{matrix}
 T^{-1}    & T^{-2}S  \\
 0   & 0   \\
\end{matrix} }} \right]U^\ast
 =
 U\left[ {{\begin{matrix}
   T ^\ast T ^{-1}   &   T ^\ast T^{-2}S  \\
  B_2^\ast  T^{-1}   & B_2^\ast T^{-2}S\\
\end{matrix} }} \right]U^\ast .
 \end{align}
 Therefore,
$A A^{\tiny\textcircled{W}}=B A^{\tiny\textcircled{W}}$
and
 $A^\ast A^{\tiny\textcircled{W}} = B^\ast A^{\tiny\textcircled{W}}$
 if and only if
\begin{align}
\label{4-3}
B_1=T ,
B_3=0,
B_2= S,
\mbox{   and\ }
B_4 \mbox{   is arbitrary},
 \end{align}
 that is,
$A$ and  $B$ satisfy
$A A^{\tiny\textcircled{W}}=B A^{\tiny\textcircled{W}}$
and
 $A^\ast A^{\tiny\textcircled{W}} = B^\ast A^{\tiny\textcircled{W}}$
if and only if
there exists a  unitary matrix $ {U}$ such that
\begin{align}
\label{4-2}
A = { U}\left[ {{\begin{matrix}
 T  & S  \\
 0   & N   \\
\end{matrix} }} \right] {U}^\ast  ,
\ \
B =  { U}\left[ {{\begin{matrix}
 T  & S  \\
 0   & B_4  \\
\end{matrix} }} \right] {U}^\ast  ,
 \end{align}
where
 $N $ is nilpotent,
 $T$ is  non-singular
 and
 $B_4$  is arbitrary. 
Therefore, by applying Lemma \ref{Lemma-1},
 we derive
one  characterization  of the core-EP order.
\begin{theorem}
\label{Theorem-4-1}
Let $A,B \in \mathbb{C}_{n,n}$.
Then
$A \mathop \le \limits^{\tiny\textcircled{\dag}} B$
if
and only if
$$
AA^{\tiny\textcircled{W}}=BA^{\tiny\textcircled{W}}
{\ \rm and\ }
A^\ast A^{\tiny\textcircled{W}} = B^\ast A^{\tiny\textcircled{W}}.$$
\end{theorem}

\section*{ Acknowledgements}.
The first author was supported partially by
the National  Natural Science Foundation of China [grant number 11401243]
and
China Postdoctoral Science Foundation [grant number 2015M581690].
The second author was supported partially by
the National Natural Science Foundation of China  [grant number 11371089]
and
the Natural Science Foundation of Jiangsu Province  [grant number  BK20141327].

\end{document}